\documentclass[a4paper,11pt,reqno]{amsart}
\usepackage{epsfig,url,paralist}
\usepackage{amssymb}
\usepackage[normalem]{ulem}
\usepackage{fullpage,dynkin-diagrams}
\usepackage{amsmath}
\usepackage{tikz}
\usepackage{fullpage}
\usepackage{hyperref,verbatim}
\usepackage[capitalize]{cleveref}
\usepackage{graphicx}
\usepackage{setspace,multirow}
\usepackage{enumitem,lineno}
\setlist{nolistsep}
\usepackage{lscape}
\usepackage{calrsfs}            

\newtheorem{theo}{Theorem}

\numberwithin{equation}{section}

\newtheorem{theorem}{Theorem}[section]

\newtheorem{theorem*}{Main Result}
\newtheorem{lem}[theorem]{Lemma}
\newtheorem{coro}[theorem]{Corollary}
\newtheorem{cor*}[theorem*]{Corollary}

\newtheorem{prop}[theorem]{Proposition}
\theoremstyle{definition}
\newtheorem{defn}[theorem]{Definition}
\newtheorem{defs}[theorem]{Definitions}

\newtheorem{defsfacs}[theorem]{Definitions \& Facts}

\newtheorem{remark}[theorem]{Remark}

\numberwithin{theorem}{section}
\usepackage{fancyhdr}
\def\<{\langle}
\def\>{\rangle}
\newcommand{\Res}{\mathsf{Res}}

\newcommand{\proj}{\mathsf{proj}}
\newcommand{\PGL}{\mathsf{PGL}}

\newcommand{\PSL}{\mathsf{PSL}}

\newcommand{\PG}{\mathsf{PG}}

\newcommand{\K}{\mathbb{K}}
\newcommand{\F}{\mathbb{F}}

\newcommand{\N}{\mathbb{N}}

\begin{document}
\title{A characterisation of lines in finite Lie incidence geometries of classical type}
\author{Sira Busch \and Hendrik Van Maldeghem}
\address{Sira Busch\\ Department of Mathematics, M\"unster University, Germany}
\email{s\_busc16@uni-muenster.de}
\address{Hendrik Van Maldeghem \\ Department of Mathematics, statistics and Computer Science, Ghent University, Belgium} \email{Hendrik.VanMaldeghem@UGent.be}
\thanks{The first author is funded by the Claussen-Simon-Stiftung and by the Deutsche Forschungsgemeinschaft (DFG, German Research Foundation) under Germany's Excellence Strategy EXC 2044 --390685587, Mathematics M\"unster: Dynamics--Geometry--Structure. This work is part of the PhD project of the first author.}

\begin{abstract}
We consider any classical Grassmannian geometry $\Gamma$; that is, any projective or polar Grassmann space. Suppose every line in $\Gamma$ contains $s+1$ points. Then we classify all sets of points in $\Gamma$ of cardinality $s+1$, with the property, that no object of opposite type in the corresponding building, is opposite every point of the set. It turns out that such sets are either lines, or hyperbolic lines in symplectic residues, or ovoids in large symplectic subquadrangles of rank~2 residues in characteristic~2. This is a far-reaching extension of a famous and fundamental result of Bose \& Burton from the 1960s. We describe a new way to classify geometric lines in finite classical geometries and how our results correspond to blocking sets.
\end{abstract}
\maketitle

\section{Introduction}

A notion of great interest in finite geometry are \emph{blocking sets}. They have applications in the theory of communication systems, coding theory and cryptography (see \cite{Holder:01}). Originally, blocking sets were defined for projective planes as a set of points in the plane that every line intersects. More generally, a blocking set can be defined as a set of points in a projective space that every hyperplane intersects. In this context, instead of talking about the cardinality of a blocking set, we will talk about its \emph{size}, but we will mean the same. The fundamental starting point to investigate the structure of blocking sets for projective planes is a result of Bose \& Burton \cite{Bos-Bur:66} that states, in modern terminology, that for a projective plane $\Gamma$, in which every line has exactly $c$ points, any blocking set of size at most $c$ in $\Gamma$ has size exactly $c$ and is a line. In this article, we want to generalise the notion of a blocking set to all \emph{finite classical geometries} $\Gamma$ and classify those that contain the same number of points as a line in $\Gamma$. We will use the fact that every such geometry corresponds to a \emph{spherical building} and we will talk more about that shorty. Using buildings, the result of Bose \& Burton can be rephrased as follows: For a projective plane $\Gamma$, in which every line has exactly $c$ points, any set of at most $c$ points that does not admit a common opposite, contains exactly $c$ points and is a line. We will now state the main result of this article.

\smallbreak

\textbf{Main Result A.} \emph{For a finite classical geometry $\Gamma$, in which every line has exactly $c$ points, any set of at most $c$ points that does not admit a common opposite, contains exactly $c$ points and is either a line, a hyperbolic line in a symplectic residue, or an ovoid in a proper ideal symplectic subquadrangle in characteristic~$2$ of a residue.}

\medbreak

Knowing the goal, we will now proceed to shed some light on the most relevant concepts that we will use. First, we will specify what we mean by \emph{finite classical geometries}. In its simplest form, a \emph{(point-line) geometry} consists of a set of \emph{points} and a set of subsets of this point set, called the set of \emph{lines}. The ``prototypes'' of point-line geometries are \emph{projective planes}, which can be defined axiomatically as point-lines geometries that satisfy three  axioms, namely, 
\begin{enumerate}[label=(\roman*)]
\item every pair of distinct points is contained in a unique line,
\item every pair of distinct lines has a unique point in common,  
\item and there exist four points of which no three are contained in a common line. 
\end{enumerate}More generally, we define a \emph{projective space} as a point line geometry, where
\begin{enumerate}[label=(\roman*)]
\item the points are the $1$-dimensional subspaces (we will say \emph{$n$-spaces} for $n$-dimensional subspaces in the following) of a vector space $V$ of dimension $n$ at least $3$, 
\item and the lines are the $1$-spaces contained in a given $2$-space. 
\end{enumerate}
This point-line geometry is denoted by $\PG(V)$ and said to have dimension $\dim V-1$. For $n=3$, we obtain examples of projective planes as defined above. 

\smallbreak

The automorphism group of $\PG(V)$ is the group of permutations of the point set, which map each line bijectively to some other line and is denoted by $\mathsf{P\Gamma L}(V)$.  It can be obtained from the group $\mathsf{\Gamma L}(V)$ of bijective semi-linear transformations of $V$ by factoring out the centre. Therefore, a projective space is a suitable object for studying the groups $\mathsf{P\Gamma L}(V)$, $\PGL(V)$ and $\PSL(V)$ (which are subgroups of $\mathsf{P\Gamma L}(V)$). 

\smallbreak

These groups are so-called \emph{classical groups}, and one might wonder, whether other classes of classical groups also admit a ``nice'' point-line geometry on which they act as automorphism groups. This gave rise to the notion of a \emph{polar space} by work of Veldkamp \cite{Vel:59-60} and Tits \cite{Tits:74} (we will define polar spaces in \cref{polarspace}).  Subsequently, Buekenhout \& Shult \cite{Bue-Shu:74}  found a simple axiom system that characterised polar spaces as point-line geometries (the axioms of Veldkamp and Tits assume the existence of projective spaces as substructures). It turns out that projective and polar spaces are the natural point-line geometries of what is nowadays called the \emph{classical groups of Lie type}. The axiom system of Tits is deduced from his more general notion of a \emph{spherical building}, which he introduced to also capture the exceptional groups of Lie type. The exact definition of a spherical building is of no importance to us (we refer to \cite{Abr-Bro:08,Tits:74}); it suffices, for this article, to know that it is a numbered simplicial complex (i.e. the vertices have types), and that there is a standard procedure to construct a point-line geometry from the set of vertices of any given type. This way, the spherical buildings that give rise to projective and polar spaces, also define other point-line geometries. 

\medbreak

If $j$ is the type of the vertices that are chosen as points, then one refers to such a geometry as the \emph{$j$-Grassmannian} of the associated projective or polar space (for details see for example \cite{Mal:24}). Usually one assigns the type $j$ to the vertices that correspond to the projective subspaces of the polar or projective space that are isomorphic to $\PG(V)$, with $\dim V=j$. The $1$-Grassmannians are simply the projective and polar spaces themselves. We will refer to these $j$-Grassmannian geometries as the \emph{classical geometries} in the following. A classical geometry is called \emph{finite}, if every line has a finite number of points. For projective and polar spaces this already implies that every line has the same number of points.

\bigbreak

A special feature in the theory of spherical buildings is the notion of \emph{opposition}. When translated to projective and polar spaces, we get the following.
\renewcommand\labelitemi{\tiny$\bullet$}
\begin{itemize}
\item Two subspaces of a projective space are opposite if, and only if, they are complementary \\(that is, they are disjoint and together they generate the whole projective space). 
\item Two projective subspaces inside a polar space are opposite if, and only if, no point of their union is collinear to all points of that union. 
\end{itemize}
This opposition relation induces an opposition relation on the types of the vertices of the corresponding spherical building. In a projective space $\PG(V)$ with $\dim V=n$, subspaces of type $j$ are opposite subspaces of type $n-j$. In polar spaces, the opposition relation on the types is the identity (or almost the identity.\footnote{Certain polar space admit a so-called \emph{oriflamme geometry} and then the opposition relation is slightly different. For details we refer to \cref{oriflam}.}) Hence, in particular, we can speak of \emph{subspaces (or objects) of opposite type}.  

\medbreak

The central question of \cite{Kas-Mal:13} is, whether (perhaps partial) knowledge of this opposition relation determines the lines of any classical geometry. This lead the authors of \cite{Kas-Mal:13} to the notion of a \emph{geometric line}, which is a set of points in a classical geometry, such that each object of opposite type is \emph{not} opposite
\begin{itemize}
\item[--]  either all points,
\item[--] or exactly one point 
\end{itemize}
of the set. 
Such sets are classified in \cite{Kas-Mal:13} for all classical geometries. In the particular case of a projective plane $\Gamma$,  lines are opposite points and here, a geometric line is a set of points, such that each line either (a) contains that set, or (b) has a unique point with it in common. 
\smallbreak
Suppose $\Gamma$ is a finite projective plane. Then each line has a constant number of points, say $c$. Condition (a) implies that a geometric line has at most $c$ points, and (a) and (b) together imply that each line intersects a geometric line in at least one point, or, in other words, no line is opposite all points of the geometric line. The latter condition is, again, the defining property of blocking sets in projective planes. It turns out that in projective planes (and also in projective spaces) lines and geometric lines are just equivalent.
\smallbreak

In the finite case, the definition of a blocking set of a projective plane (and also of any projective space) with the size of a line is ostensibly weaker than the definition of a geometric line. If we define, for an arbitrary finite classical geometry, a blocking set as a set of points admitting no global opposite, then, again, the definition of a blocking set with the size of a line in a classical geometry is much weaker than that of a geometric line. However, in this article, we show that in the odd characteristic case these notions are the same --- in characteristic~$2$ there are exceptions. This is the content of a second main result.
 \smallbreak
\textbf{Main Result B.} \emph{In a finite classical geometry $\Gamma$ with $c$ points on each line, and with $c$ even, a set of at most $c$ points is 
a geometric line 
 if, and only if, no object of the opposite type of a point is opposite each point of the set.}

\medbreak

The condition on $c$ can be lifted if $\Gamma$ is neither the $n$-Grassmannian of the polar space of rank $n\geq 2$ associated to an elliptic quadric, nor the $(n-1)$-Grassmannian of a small Hermitian polar space of rank $n\geq 2$  (see \cref{defpolarspaces}). 

\bigbreak

We can also motivate our main results slightly differently as follows. It is known that, in a polar space with $s+1$ points per line, any set of $s$ points admits an opposite point (this is Exercise 2.13 in \cite{Mal:24}). That means there exists a point, which is non-collinear to any given set of at most $s$ points. This is not true for $s+1$, as lines are counterexamples. In this minimal case, it is natural to ask for a classification of all counterexamples. This can be phrased as a Segre-like or extremal problem as follows: If a set $T$ of points does not admit a common opposite, then $|T|$ is at least $s+1$; what happens, if equality occurs? 
\medbreak

Yet another motivation, which in fact initiated this work, was a problem arising in \cite{Bus-Sch-Mal:25}. In order to decompose an arbitrary sequence of perspectivities into a sequence of perspectivities of a certain prescribed type, we had to find a point opposite four arbitrary given points of a certain Lie incidence geometry. This is easy, if the ground field has size at least $4$, but over the field $\F_3$, this problem gave rise to exactly the question sketched in the previous paragraph with $s=3$. The solution for $s=3$ is not much simpler than the general case, and so we proceeded to answer this question in full generality for all possible classical geometries. In a subsequent paper \cite{Bus-Mal:25}, we treat some exceptional geometries. The reason to not include the exceptional geometries in the present paper is that these geometries require different methods, the framework of parapolar spaces and much deeper understanding about Lie incidence geometries.

\medbreak

Lastly, we would like to mention some work of Cohen \& Cooperstein \cite{Coh-Coo:98}. Let $\Gamma$ be a point-line geometry. A line $L$ of $\Gamma$ is called a \emph{full line} of $\PG(V)$, if the points of $L$ coincide exactly with all points of some (unique) line of $\PG(V)$. A \emph{full projective embedding} of $\Gamma$ is a representation of $\Gamma$ as a set of points of $\PG(V)$, for some vector space $V$, such that the lines of $\Gamma$ are full lines of $\PG(V)$. In \cite{Coh-Coo:98}, Cohen \& Cooperstein describe all full embeddings of classical geometries $\Gamma$ in $\PG(V)$, which contain sets of points that form a line in $\PG(V)$, but that do not form a line in $\Gamma$. The results of \cite{Kas-Mal:13} show that these sets of points are precisely the geometric lines. 


\section{Preliminaries}

We will briefly recall the most important features of finite polar spaces (having given the definition of projective spaces in the previous section) and refer to standard textbooks such as \cite{Bue-Coh:13,Shu:11,Tits:74,Mal:24} and the chapter \cite{Coh:95} in the Handbook of Incidence Geometry for more background; also \cite{Cam:92} and Chapter~2 of \cite{Bro-Mal:22} contain a detailed introduction to the finite case. 

\subsection{Polar spaces}

First of all, we use standard terminology and notation concerning point-line geometries: points $p$ and $q$ contained in a line $L$ are called \emph{collinear}, and we write $p\perp q$, and, if $p\neq q$, $L=pq$ (in all our geometries lines are determined by an pair of their points). The set of points collinear to a point $p$ is denoted by $p^\perp$. For a set $S$ of points, the set of points collinear to each member of $S$ is denoted as $S^\perp$. A \emph{subspace} $D$ is a set of points, with the property, that each line containing two distinct points of $D$, is entirely contained in $D$. A subspace is \emph{singular}, if each pair of its points is collinear. 

\begin{defn}[\textbf{Polar space}]\label{polarspace}
A \emph{polar space} $\Gamma$ is a point-line geometry, such that
\begin{enumerate}[label=(\roman*)]
\item  $\Gamma$ contains at least one line,
\item every line in $\Gamma$ contains at least three points,
\item  for every point $p$ and every line $L$ in $\Gamma$, either exactly one, or all points of $L$ are collinear to $p$,
\item and no point is collinear to every other point.
\end{enumerate}
\end{defn}
In the literature, polar spaces satisfying (iv) are sometimes called \emph{non-degenerate}.

\begin{defs}[\textbf{Rank}, \textbf{maximal singular subspaces}, \textbf{generalised quadrangles}]
If in a given polar space $\Gamma$, there do not exist three mutually collinear points that are not contained in a common line, we say that the \emph{rank of $\Gamma$ is $2$}. If there exists a natural number $r\geq 3$, such that no singular subspace is isomorphic to an $r$-dimensional projective space, but there exist singular subspaces isomorphic to an $(r-1)$-dimensional projective space, then we say that the \emph{rank of $\Gamma$ is $r$}. The singular subspaces isomorphic to $(r-1)$-dimensional projective spaces are then also called \emph{maximal singular subspaces} (sometimes they are also referred to as \emph{generators}). Every other singular subspace is contained in a maximal one and is, consequently, a projective space. The singular subspaces isomorphic to $(r-2)$-dimensional projective spaces are called \emph{submaximal}. Polar spaces of rank $2$ are also called \emph{generalised quadrangles} (see \cite{FGQ} for background on the finite case).  
\end{defs}

\begin{defsfacs}[\textbf{Order}]
We say that a finite polar space of rank $r\geq 2$ \emph{has order} $(s,t)$, if every line carries exactly $s+1$ points, and every submaximal subspace is contained in exactly $t+1$ maximal singular subspaces. A projective space has \emph{order $s$}, if every line contains exactly $s+1$ points. In that case, and if it is classical, it is defined over the finite field with $s$ elements. Every singular subspace of a polar space of order $(s,t)$ is a projective space of order~$s$.
\end{defsfacs}

\begin{defn}[\textbf{Ideal subquadrangle}]\label{idealsub}
A \emph{subquadrangle} $\Gamma'$ of a generalised quadrangle $\Gamma$ is a subset $X$ of points which, endowed with the intersections of $X$ with the lines of $\Gamma$ containing at least two points of $X$, is a generalised quadrangle.  If $\Gamma$ has order $(s,t)$ and $\Gamma'$ has order $(s',t)$, then we say that $\Gamma'$ is an \emph{ideal} subquadrangle. 
\end{defn}

\begin{defs}[\textbf{Grassmannian}, \textbf{dual polar space}]
For a finite polar space $\Gamma$ of rank $r\geq 2$ and a positive integer $i$, $1\leq i\leq r$, the $i$-Grassmannian is the geometry with point set the set of singular subspaces of dimension $(i-1)$, and the lines are the sets of such subspaces containing a given singular subspace $U$ of dimension $i-2$ (a subspace of dimension $0$ being just the empty set) and contained in a given singular subspace $W\supseteq U$ of dimension $i$ (the second condition is deleted if $i=r-1$). The $r$-Grassmannian geometry is usually called a \emph{dual polar space}. 
\end{defs}

In the following, let $V$ be a vector space over an arbitrary field $\K$. If $\K$ is the finite field with $q$ elements and has dimension $n+1$, then we denote $\PG(V)$ as $\PG(n,q)$. We assume that the reader is familiar with bilinear, quadratic and Hermitian forms. 
\begin{defs}[\textbf{Quadrics, isotropic vectors, polar spaces associated to forms}]\label{defpolarspaces} \

\begin{itemize}
\item The projective null set of a quadratic form will be called a \emph{quadric}; it is non-degenerate, if the null set is disjoint from the radical of the associated symmetric bilinear form, up to the trivial zero vector. 
\item Let $f$ be a symmetric, alternating bilinear or Hermitian form on $V$. We call a vector $v$ an \emph{isotropic vector with respect to $f$}, if $f(v,v)=0$. 
\item By the \emph{polar space associated to $f$}, we will mean a polar space $\Gamma$, such that
\begin{itemize}
\item[$\cdot$] for the isotropic vectors $v$ with respect to $f$, the $1$-spaces $\<v\>$ of $V$, viewed as points in $\PG(V)$, define the points of $\Gamma$,
\item[$\cdot$] the lines of $\Gamma$ correspond to the totally isotropic $2$-spaces -- that is, $2$-spaces $\<v,w\>$, such that $f(v,v)=f(v,w)=f(w,w)=0$ --- assuming these exist. 
\end{itemize}
\end{itemize}
\end{defs}
A non-degenerate quadric containing lines, also defines a polar space. If the characteristic of $\K$ is not equal to $2$, this polar space coincides with the one associated to the corresponding symmetric bilinear form. We refer to \cite[Chapter~3]{Mal:24} for an elementary treatment. 
\begin{defs} \
\begin{itemize}
\item A \emph{symplectic} (\emph{Hermitian}) polar space is the polar space associated to a non-degenerate alternating (Hermitian) form, or, equivalently, to a symplectic (unitary) polarity of a projective space. 
\item A \emph{parabolic quadric} is a non-degenerate quadric in a projective space of even dimension at least $4$ (remember we restrict to the finite case here). 
\item An \emph{elliptic quadric} is a non-degenerate quadric in a projective space of odd dimension $2n+1\geq3$ (remember we restrict to the finite case here) not containing subspaces of dimension $n$.  
\item A \emph{hyperbolic} polar space is a polar space arising from a non-degenerate hyperbolic quadric, that is, a non-degenerate quadric in $\PG(2n-1,q)$ of rank $n+1$ (so, maximal singular subspaces have projective dimension $n$). These polar spaces have order $(q,1)$, and their maximal singular subspaces fall into two natural classes. Two maximal singular subspaces $M,M'$ belong to the same class if and only if their intersection has even codimension in both, that is, $\dim M-\dim(M\cap M')\in2\mathbb{Z}$. In this case, the $n$-Grassmannian geometry has lines of size $2$ and we replace it by two \emph{half spin geometries}, which each take one of the natural classes of maximal singular subspaces as point sets, and a generic line is the set of maximal singular subspaces of the given class containing a given singular subspace of dimension $n-3$ (see \cite[Section 6.4 \& 9.4.3]{Mal:24}). 
\end{itemize}
\end{defs}
In the finite case, there are projectively unique Hermitian polar spaces in $\PG(n,q^2)$, for each $n\geq 3$ and $q$ any prime power. If $n$ is even their order is $(q^2,q^3)$, and if $n$ is odd, then their order is $(q^2,q)$. We will refer to the latter as the \emph{small Hermitian polar spaces}. 
\begin{defn}[\textbf{Hyperbolic line}]
A \emph{hyperbolic line} in a polar space $\Delta$ is a set of points collinear to every point collinear to two given non-collinear points.  In other words: It is a set of points of the form $(\{x, y\}^{\perp})^{\perp}=\{x,y\}^{\perp\perp}$, where $x$ and $y$ are two non-collinear points. It is called \emph{large} if it coincides with $\{u,v\}^\perp$, for $u$ and $v$ any two distinct points of $\{x,y\}^\perp$.  
\end{defn}

As mentioned before, a crucial notion for the present paper is that of opposition, and for polar spaces, two singular subspaces are opposite, if no point of either of them is collinear to all points of both of them. We note that it automatically follows that they have the same projective dimension (see \cite[Corollary~1.4.7]{Mal:24}). 

\begin{defn}[\textbf{Regulus}]
In a polar space associated to a quadric, the set of lines meeting two given opposite lines is called a \emph{regulus}. 
\end{defn}

\begin{defn}[\textbf{Ovoid, spread}]
An \emph{ovoid} of a polar space is a set of points with the property that every line contains exactly one point of that set. A \emph{spread} is a set of generators that partitions the point set. It is straight forward that the number of points of an ovoid and the number of lines of a spread of a generalised quadrangle of order $(s,t)$ is equal to $1+st$ (see 1.8.1 of \cite{FGQ}). 
\end{defn}

\subsection{Detailled description of the main results}

We first phrase the assumptions of our main results uniformly in building theoretic terms, and only afterwards we specify them for projective and polar spaces. Concerning the building theoretic notions, we refer to the standard references \cite{Abr-Bro:08,Tits:74} and the excellent survey \cite{Coh:95}. We recall that we see buildings as simplicial complexes in which the maximal simplices are called \emph{chambers} and the submaximal ones (or next-to-maximal ones) \emph{panels}. A building is \emph{thick} (\emph{thin}, respectively), if every panel is contained in at least three (exactly two, respectively) chambers. We will call a panel \emph{$s$-thick}, if it is contained in precisely $s+1$ chambers. An \emph{apartment} is a thin subcomplex containing chambers. A building is called \emph{spherical}, if its apartments have a fine number of vertices. The automorphism group of a single apartment of a building is a Coxeter group, to which a Coxeter diagram can be attached, and we use the Bourbaki labeling of the types \cite{Bou:68}, and for projective and polar spaces, this agrees with how we defined types in the introduction (except for the hyperbolic case). For each thick building $\Delta$, say of type $\mathsf{X}_n$, and each type, say $i\in\{1,2,\ldots,n\}$, there is a unique point-line geometry, in which the points are the vertices of type $i$ of $\Delta$ and the lines are the sets of vertices of type $i$, completing a given panel -- obtained from a chamber by removing the vertex of type $i$ -- to a chamber. This geometry is called a \emph{Lie incidence geometry of type $\mathsf{X}_{n,i}$}.  

\begin{theo} \label{A} If in an irreducible thick finite classical spherical building $\Delta$ of type $\mathsf{X}_n$, for $n\geq 2$, the panels of cotype $\{i\}$ are $s$-thick, then every set $T$ of $s+1$ vertices of type $i$ of $\Delta$ admits a common opposite vertex except in precisely the following six cases.
\begin{compactenum}[$(1)$]
\item The set $T$ is a  line in the corresponding Lie incidence geometry of type $\mathsf{X}_{n,i}$.
\item $\Delta$ is a generalised quadrangle of order $(s,t)$ and $T$ is an ovoid in a subquadrangle of order $(s/t,t)$, or $\Delta$ is a generalised quadrangle of order $(t,s)$ and $T$ is a spread in a subquadrangle of order $(t,s/t)$.
\item  $\Delta$ is the building corresponding to a symplectic polar space of rank at least $3$ and $T$ is a hyperbolic line in the residue of a simplex of type $\{1,2,\ldots,i-1\}$. 
\item $\Delta$ is the building corresponding to a parabolic quadric $\Gamma$ and $T$ is a set of generators of $\Gamma$, such that
\begin{itemize}
\item[$\cdot$] all members of $T$, viewed as a subspaces of $\Gamma$, contain a common singular subspace $U$ of codimension $2$,
\item[$\cdot$] in the residue of $U$, which is a generalised quadrangle associated to a quadric, the set $T$ is a regulus. 
\end{itemize}
\item $\Delta$ is the building corresponding to a small Hermitian polar space of rank at least $3$ in characteristic~$2$, $i=n-1$ and $T$ is an ovoid in a symplectic  subquadrangle with order $(t,t)$ of the residue of a simplex of type $\{1,2,\ldots,n-2\}$.
\item $\Delta$ is the building corresponding to an elliptic quadric of rank at least $3$ in characteristic~$2$, $i=n$ and $T$ is a spread in a symplectic  subquadrangle with order $(s,s)$ of the residue of a simplex of type $\{1,2,\ldots,n-2\}$.
\end{compactenum}
\end{theo}

In Case (2), if $s=t$, then $T$ is either a hyperbolic line in $\Delta$, or in the dual of $\Delta$; the latter is the generalised quadrangle obtained from $\Delta$ by interchanging the points and the lines. 

For projective spaces, this theorem can be stated as follows, with the convention that the empty projective subspace has dimension $-1$. 

\begin{coro}\label{Anmain}
Let $0\leq k<n$ be integers, and let $q$ be a prime power. Let $T$ be a set of $q+1$  $k$-dimensional subspaces of $\PG(n,q)$. If no $(n-k-1)$-space is disjoint from each member of $T$, then there exist a $(k-1)$-space $U$ and a $(k+1)$-space $W$, such that $T$ coincides with the set of $k$-spaces containing $U$ and contained in $W$.
\end{coro}

For polar spaces we have the following formulation (excluding the rather trivial case of a grid, which corresponds to a reducible building of type $\mathsf{A_1\times A_1}$).

\begin{coro}
Let $\Gamma$ be a polar space of rank $r$ at least $2$ and order $(s,t)$, with $t>1$ if $r=2$. Let $T$ be either\begin{itemize} \item a set of $s+1$ singular subspaces of $\Gamma$ of dimension $k\leq r-2$, or \item only if $t>1$, a set of $t+1$ maximal singular subspaces (and we set $k=r-1$), or \item only if $t=1$, a set of $s+1$ maximal singular subspaces of the same natural class (and again $k=r-1$). \end{itemize} Then there exists a singular subspace of dimension $k$ opposite each member of $T$, except if
\begin{compactenum}[$(i)$]
\item $k\leq r-2$ and all members of $T$ contain a given $(k-1)$-dimensional subspace and are contained in a given $(k+1)$-dimensional singular subspace;
\item $k=r-1$, $\Gamma$ is not hyperbolic and all members of $T$ contain a given $(r-2)$-dimensional subspace;
\item $k\leq r-2$, $\Gamma$ is symplectic, and all members of $T$ contain a given $(k-1)$-dimensional subspace in the residue of which they form a hyperbolic line;
\item $k=r-1$, $\Gamma$ is either parabolic or hyperbolic, and all members of $T$ contain a given $(r-3)$-dimensional subspace in the residue of which they form a regulus;
\item $k=r-2$, $\Gamma$ is a small Hermitian polar space over a field of characteristic~$2$ and all members of $T$ contain a given $(r-3)$-dimensional subspace in the residue of which they form an ovoid in a symplectic subquadrangle of order $(t,t)$;  
\item $k=r-1$, $\Gamma$ is elliptic over a field of characteristic~$2$ and all members of $T$ contain a given $(r-3)$-dimensional subspace in the residue of which they form a spread in a symplectic subquadrangle of order $(s,s)$; 
\item $r=2$ and $T$ is either an ovoid in a subquadrangle of order $(s/t,t)$, or a spread in a subquadrangle of order $(s, t/s)$.  
\end{compactenum}
\end{coro}

In the last case, if $s=t$, then $T$ is half of a subquadrangle of order $(1,s)$ or $(s,1)$, respectively. If $s\neq t$, then \cref{ovoidcaseproof} guarantees that $(vii)$ produces examples.  

There is a rather intriguing corollary to our main results --- that we stated as our second main result in the introduction --- which makes a connection with the notion of a geometric line. We restrict ourselves to the case of rank at least $3$.

\begin{theo}\label{B}
 If in an irreducible thick finite classical spherical building $\Delta$ of type $\mathsf{X}_n$, $n\geq 3$, the panels of cotype $\{i\}$ are $s$-thick, $s$ odd, then a set $T$ of vertices of type $i$ of $\Delta$ is a geometric line in the $i$-Grassmannian geometry (of type $\mathsf{X}_{n,i}$) if, and only if, it has size $s+1$ and does not admit a common opposite vertex. The condition on $s$ can be lifted if $\Delta$ does not correspond to either a small Hermitian polar space if $i=n-1$, or an elliptic quadric if $i=n$.
\end{theo}

Interestingly, our proofs do not use the notion of a geometric line. The proof of \cref{B} consists of observing that both notions provide almost always the same objects; both after arguing for some pages. Only in \cref{polarpoints}, where we treat the case of points in a polar space, we are able to use the notion of geometric lines to obtain our classification. The proof is not as direct as one could hope. However, we also provide a shorter alternative proof of the same result in the special case, where the polar space is related to a quadric. There, we do not use the notion of geometric lines, see \cref{hyppoints}. 

Another corollary is the following.

\begin{coro}
 If in an irreducible thick finite classical spherical building $\Delta$ of type $\mathsf{X}_n$, $n\geq 2$, the panels of cotype $\{i\}$ are $s$-thick, then every set $T'$ of $s$ vertices of type $i$ of $\Delta$ admits a common opposite vertex.
\end{coro}

\begin{proof}
We can always complete $T'$ to a set $T$ of $s+1$ vertices by adding a vertex such that $T$ is not a set as in the conclusion of \emph{\cref{A}}. 
\end{proof}

\begin{remark}
There is a notion of \emph{split building}, which is essentially the building associated to a Chevalley group (or split group of Lie type, using the terminology of \cite{Car:89}), see Chapter 15 of \cite{Car:89}. Without defining this in general, we mention that, in the finite classical case, this concerns the projective spaces, the symplectic polar spaces (which are then said to be of type $\mathsf{C}_n$), the parabolic polar spaces (type $\mathsf{B}_n$) and the oriflamme complexes of hyperbolic polar spaces (type $\mathsf{D}_n$; see \cref{oriflam} for a precise definition). These have not just a Coxeter diagram attached, but a \emph{Dynkin diagram}, where nodes correspond to fundamental roots of a root system. Then the cases in the conclusions of \cref{A}, where the set $T$ is not a line in the Lie incidence geometry of type $\mathsf{X}_{n,i}$, occur precisely when $i$ represents a short root in the root system corresponding to the Dynkin diagram. This behaviour will sustain in the exceptional case (see \cite{Bus-Mal:25}). 
\end{remark}


\section{Proofs of the main results}

First, we recall the following extension of Theorem~3.30 of \cite{Tits:74}. For a proof, see Proposition 8.2 of \cite{Bus-Sch-Mal:25}. 

\begin{prop}\label{3.30}
If every panel of a spherical building is contained in at least $s+1$ chambers, then every set of $s$ chambers admits an opposite chamber.  
\end{prop}


\subsection{Projective spaces---Type $\mathsf{A}_n$}
The following lemma proves \cref{A} for buildings of type $\mathsf{A}_n$, that is, for projective spaces and its Grassmannians.  
\begin{lem}\label{An}
If no $(n-k-1)$-space is disjoint from each member of a set $T$ of $q+1$ projective $k$-spaces of $\PG(n,q)$, then there exist a $(k-1)$-space $U$ and a $(k+1)$-space $W$ such that $T$ coincides with the set of $k$-spaces containing $U$ and contained in $W$.
\end{lem}

\begin{proof}
We first do the cases $n\leq 3$ and then proceed by induction on $n$. If $k=0$, then the assertion follows directly from the main result of \cite{Bos-Bur:66}. Dually, the case $k=n-1$ follows. Whence the case $n=2$. Now suppose $n=3$ and $k=1$. Suppose at least two members of $T$ intersect in a point, say, $L_1\cap L_2=\{p\}$, $L_1,L_2\in T$. Since we may suppose that $T$ is not a planar line pencil, there is some point $x\in\<L_1,L_2\>$ not contained in any member of $T$. In the residue of $x$, then lines obtained from $T$ form a set of at most $q$ members (since $L_1$ and $L_2$ define the same line), and hence, by \cite{Bos-Bur:66} again,  we find a line $K$ through $x$ disjoint from all members of $T$. Now suppose every pair of members of $T$ is disjoint. Pick a point $x$ not on any member of $T$ (that is possible since there are $q^3+q^2+q+1$ points and only $(q+1)^2$ on members of $T$). Pick $L_1,L_2\in T$ arbitrarily. Then there exists a unique line $K$ through $x$ intersecting both $L_1$ and $L_2$ non-trivially.    Since $K\setminus\{x\}$ contains $q$ points, there exists a member $L_3\in T$ not meeting $K$. Therefore, in the residue of $x$, not all lines corresponding to the members of $T$ go through a common point. With that, again, there exists a line through $x$ not intersecting any member of $T$. This shows the case $n=3$.

We proceed by induction. By duality, we may assume $2k+1\leq n$ and we consider two different cases.

\begin{compactenum}[$(1)$]
\item \emph{Suppose each pair of members of $T$ intersects in a $(k-1)$-space.} Then it is easy to see that there are again two cases.
\begin{compactenum}[$(i)$]
\item \emph{The members of $T$ contain a common $(k-1)$-space $U$}. Then we may assume they are not all contained in a common $(k+1)$-space. Intersecting the situation with a hyperplane $H$ not containing $U$, we obtain $q+1$ $(k-1)$-spaces in $\PG(n-1,q)$ all going through the same $(k-2)$-space, but not contained in a common $k$-space. Applying induction we obtain an $((n-1)-(k-1)-1)$-space $Z\subseteq H$ disjoint  from each member of $T$. 
\item \emph{The members of $T$ are contained in a common $(k+1)$-space $W$}. Here, we may assume that they do not contain a common $(k-1)$-space. Since $k+1\leq n-1$, we can apply induction and find a point $p\in K$ not contained in any member of $T$. Let $C$ be an $(n-2-k)$-space complementary to $W$. Then $\<p,W\>$ is an $(n-1-k)$-space disjoint from all members of $T$.
\end{compactenum}
\item \emph{Some pair $\{A_1,A_2\}$ of members of $T$ intersect in at most a $(k-2)$-space}.  If either $2k+1\leq n-1$, or not all pairs in $T$ are disjoint, then we can find a point $x$ outside the span $\<A_1,A_2\>$ of two members $A_1,A_2\in T$, with $k+2\leq\dim\<A_1,A_2\>\leq n-1$, and not lying in any member of $T$ (use a simple count).  It follows that we can apply induction in the residue of $x$ and obtain an $(n-k-1)$-space through $x$ disjoint from each member of $T$. 

So we may assume that $2k+1=n$ and all members of $T$ are pairwise disjoint. Then the proof is similar to the last arguments of the case $(k,n)=(1,3)$ above. \qedhere
\end{compactenum}
\end{proof}


\subsection{Opposition in polar spaces}
We now interrupt the proof to review some characterisations of opposition in polar spaces. Since some of our proofs will be inductive, we must recognise opposite singular subspaces from \emph{locally opposite} subspaces. 

\begin{defn}[\textbf{Residues, local opposition}]
Let $\Delta$ be a polar space. The \emph{residue} $\Res_\Delta(U)$ of a singular subspace $U$ of $\Delta$ is a subpolar space of $\Delta$, such that
\begin{itemize}
\item the singular subspaces of $\Delta$ of dimension $1+\dim U$, which contain $U$, give rise to the points of $\Res_\Delta(U)$, 
\item the lines are defined by the singular subspaces of $\Delta$ of dimension $2+\dim U$, containing $U$ in the natural way.
\end{itemize}
Two singular subspaces $U,W$ of $\Delta$ are called \emph{locally opposite (at $U\cap W$)}, if no point of $(U\cap W)\setminus (U\cap W)$ is collinear to all points of $U\cup W$. This is equivalent to $U$ and $W$ being opposite in $\Res_\Delta(U\cap W)$. If it is clear, what the ambient polar space is supposed to be, we write $\Res(U)$ instead of $\Res_\Delta(U)$. 
\end{defn}

We have the following local-to-global characterisation.

\begin{lem}\label{lemma}
Let $U,W$ be two singular subspaces of some polar space $\Delta$. Let $A\subseteq U$ be a subspace. Set $B:=A^\perp\cap W$. Then the singular subspace $S$, spanned by $A$ and $B$, is locally opposite $U$ at $A$, and locally opposite $W$ at $B$ if, and only if, $U$ and $W$ are opposite in $\Gamma$. 
\end{lem}

\begin{proof}
Suppose first that $S$ is locally opposite $U$ at $A$ and locally opposite $W$ at $B$. Then 
no point of $W\setminus B$ is collinear to all points of $U$, since no such point is collinear to all points of $A$. No point of $B\subseteq S$ is collinear to all points of $U$, since $S$ is locally opposite $U$ at $A$. No point of $U\setminus A$ is collinear to all points of $B$, as such a point would otherwise be collinear to all points of $S$ (recalling that $A$ and $B$ generate $S$), contradicting the local opposition of $S$ and $U$ at $A$. Finally, no point of $A$ is collinear to all points of $W$, since $W$ is locally opposite $S$ at $B$. Hence $U$ and $W$ are opposite and the converse is proved similarly. 
\end{proof}

\begin{defn}[\textbf{Projections}]
Let $U,W$ be two singular subspaces of some polar space $\Delta$. We denote the set of points $W^{\perp} \cap U$ by $\proj_{U} (W)$ and call it the \emph{projection of $W$ onto $U$}. 
\end{defn}

We note that generators are opposite if, and only if, they are disjoint. If $U$ and $W$ are singular subspaces of a polar space with the same dimension, then they are opposite if, and only if, no point of $U$ is collinear to all points of $W$ (see Corollary~1.4.7 of \cite{Mal:24}). We will use these without reference. 


\subsection{Oriflamme geometries---Type $\mathsf{D}_n$}\label{oriflam}

The main purpose of this section is to prove \cref{A} for the extremal type $n$ for buildings of type $\mathsf{D}_n$. We argue in the corresponding geometries, which are coined \emph{oriflamme geometries} in \cite{Tits:74}. They arise from hyperbolic polar spaces of rank $r$ by forgetting the singular subspaces of dimension $r-2$ and splitting up the maximal singular subspaces into two natural classes (with natural incidence, except that two maximal singular spaces from different classes are incident, if they intersect in a subspace of dimension $r-2$).  Note that opposition now does not necessarily act trivially on the types: since maximal singular subspaces are opposite if, and only if, they are disjoint, opposition acts trivially on the types if, and only if, $r$ is even. If $r$ is odd, the opposites of a maximal singular subspace of one natural class are contained in the other natural class. 

We should first prove \cref{A} for points of the oriflamme geometries. We provide a general proof for points of polar spaces of any type below  (see \cref{polarpoints}). We prove this separately for oriflamme geometries nonetheless, because this proof does not use the notion of geometric lines and it is simpler in nature. It holds for all polar spaces associated to quadrics, but we will not need this. 

\begin{lem}\label{hyppoints}
If every line of a hyperbolic quadric $Q$ of rank at least $3$ contains exactly $s+1$ points, then there exists a point non-collinear to each point of an arbitrary set $T$ of $s+1$ (distinct) points, except if these points are contained in a single line. 
\end{lem}

\begin{proof}
Let $T=\{p_0,\ldots,p_s\}$ be a set of $s+1$ distinct points of $Q$, and suppose first that $p_0$ is collinear to $p_1$. Since not all points $p_0,p_1,\ldots,p_s$ are contained in one line, we find a point $b\in p_0p_1$ not contained in $T$. Suppose now that $p_0$ and $p_1$ are not collinear, then, since $\{p_0,p_1\}^{\perp\perp}=\{p_0,p_1\}$, we find a point $b\in \{p_0,p_1\}^\perp$ not collinear to $p_2$. 

In any case, the point $b$ has the property that the number of lines joining $b$ to a collinear point in $T$ is at most $s$. Applying \cref{3.30} we find a line $L$ through $b$, such that no point of $T$, that is collinear to $b$, is collinear to all points of $T$. Consequently, every point $p_i$ of $T$ is collinear to a unique point $p'_i$ of $L$. Since $p_0'=p_1'$ and $|L|=s+1$, there is at least one point $q\in L$ not collinear to any member of $T$.  
\end{proof}

\begin{lem}\label{hypMSS}
If every line of a hyperbolic quadric $Q$ with Witt index $d\geq 3$ contains exactly $s+1$ points, then there exists a maximal singular subspace opposite each member of an arbitrary set $T$ of $s+1$ (distinct) maximal singular subspaces of common type, except if these maximal singular subspaces contain a common singular subspace of codimension $2$ in each.  
\end{lem}
\begin{proof}
For $d=3$, this is \cref{An}, Case $\PG(3,s)$. For $d=4$, triality yields the result using \cref{hyppoints}. So suppose $d\geq 5$. We argue by induction on $d$ and consider various possibilities.  
\begin{compactenum}[$(i)$]

\item \emph{All members of $T$ contain a common point $p$.}  
In this case, we apply induction in the residue $\Res_{\Delta}(p)$ at $p$ and obtain a maximal singular subspace $M$ through $p$ intersecting every member of $T$ in just $p$. Let $p'$ be a point of $\Delta$ opposite $p$ and let $M'$ be the unique maximal singular subspace of $\Delta$ containing $p'$ and adjacent to $M$, that is, intersecting $M$ in a subspace of dimension $n-2$. Then clearly $M'$ is disjoint from each member of $T$.

\item \emph{At least two members of $T$ intersect, but they do not all share a common point.} Set $T=\{V_0,V_1,\ldots, V_s\}$. Let $p$ be a point of $\Delta$ contained in at least two members of $T$, say $V_0,V_1$, but not in all of them, say $p\notin V_2$.  Let $T'$ be the set of maximal singular subspaces through $p$ obtained by taking those of $T$ containing $p$, and taking for each member $V_i$ not containing $p$ an arbitrary maximal singular subspace $V_i'$ containing $p$ and intersecting $V_i$ in a subspace of codimension $2$.  If all members of $T'$ intersected in a common subspace of dimension $d-3$, then by replacing $V_2'$ with another maximal singular subspace through $p$ intersecting $V_2$ in subspace of dimension $d-3$ (distinct from $V_2\cap V_2'$), we obtain a new set $T'$ not having that property. Hence the induction hypothesis applied to $\Res_\Delta(p)$ yields a maximal singular subspace $W$ through $p$ intersecting each member of $T'$ exactly in $p$, and hence, since $T_i$ and $T_i'$ have a subspace of dimension $d-3$ in common and $W$ intersects each member of $T$ in a subspace of even dimension (which cannot contain a plane), $W$ intersects each member $V_i$ of $T$ in exactly one point $p_i$. Since $p_0=p_1$, the number of such points is at most $s$ and so we find a hyperplane $H$ of $W$ disjoint from all $V_i\in T$. The unique maximal singular subspace $U$ distinct from $W$ and containing $H$ has the opposite type of the $V_i$ and hence is disjoint from all of them (since it cannot intersect any of them in at least a line as this would imply that $H$ intersects a member of $T$ non-trivially). 

\item \emph{Each pair of members of $T$ is opposite.} A simple count yields a point $b$ not contained in any member of $T$. Let $T'$ be the set of maximal singular subspaces through $b$ intersecting a given member of $T$ in a hyperplane. Suppose two members $V_0',V_1'$ of $T'$ intersect in a subspace $U$ of dimension $d-3$. Then the corresponding members $V_0,V_1$, respectively, of $T$ intersect $U$ in a subspace of dimension $d-4$, and those have mutually a subspace in common of dimension at least $d-5\geq 0$, a contradiction to the disjointness of $V_0$ and $V_1$. Hence induction yields a maximal singular subspace $M$ through $b$ intersecting each member of $T'$ in just $\{b\}$. Since all points collinear to $b$ of the union of the members of $T$ are contained in the members of $T'$, and all points of $M$ are collinear to $b$, we conclude that $M$ is disjoint from each member of $T$.\qedhere
\end{compactenum}
\end{proof}


\subsection{Polar spaces---Types $\mathsf{B}_n$, $\mathsf{C}_n$ and $\mathsf{D}_n$}

We first prove \cref{A} for type 1 vertices, that is, points of the corresponding polar space. We distinguish between rank at least 3 and rank 2. After all, in rank~2, there are additional examples not featuring in higher rank. 
\begin{lem}\label{polarpoints}
If every line of a polar space $\Delta$ of rank at least $3$ contains exactly $s+1$ points, then there exists a point non-collinear to each point of an arbitrary set $T$ of $s+1$ (distinct) points, except if these points are contained in a single line, or if they form a hyperbolic line in case $\Delta$ is a symplectic polar space.  
\end{lem}

\begin{proof}
Suppose no point of $\Delta$ is non-collinear to each point of $T$, that is, each point of $\Delta$ is collinear to some member of $T$. We prove that $T$ is a geometric line in the sense of \cite{Kas-Mal:13} and then the result follows from Lemmas~4.7 and~4.8 of \cite{Kas-Mal:13}. 

So we have to prove that every point of $\Delta$ is collinear to exactly one or all points of $T$. Since we assume that each point is collinear to at least one point of $T$, it suffices to show that, if some point is collinear to at least two points of $T$, then it is collinear to all points of $T$. Suppose for a contradiction that some point $p$  is collinear to at least two points $x,y$ of $T$, and opposite the point $z\in T$. We claim that we may assume that $p\notin T$. Indeed, suppose it is. Then we may as well assume $p=y$.  Let $L$ be the line through $x$ and $y$. Then there are at most $s$ points of $T$ contained in $L$, hence there is a point $q\in L\setminus T$. The number of lines through $q$ containing a line that also contains a point of $T$ is at most $s$, hence Exercise~2.11$(iii)$ of \cite{Mal:24} in $\Res_\Delta(q)$ yields a line $K\ni q$ not collinear to any point of $T$. Since $x$ and $y$ project to the same point on $K$, there is some point on $K$ opposite each point of $T$, a contradiction. The claim is proved. But now exactly the same argument replacing $q$ with $p$ (and using the fact that $z$ is not collinear to $p$ to obtain the assertion that there are at most $s$ lines through $p$ containing a point of $T$) again leads to a contradiction. Hence the lemma is proved.  
\end{proof}

\begin{lem}\label{quadranglepoints}
If every line of a polar space $\Delta$ of rank $2$ contains exactly $s+1$ points, and every point is contained in exactly $t+1$ lines, $t>1$, then there exists a point non-collinear to each point of an arbitrary set $T$ of $s+1$ (distinct) points, except if either these points are contained in a single line, or they form an ovoid in a subquadrangle of order $(s/t,t)$. In particular, if $s=t$, they form either a line or a  large hyperbolic line, that is, $T=T^{\perp\perp}=\{y,z\}^\perp$, for every pair of distinct points $y,z$ of $T^\perp$. 
\end{lem}

\begin{proof} 
Suppose that no point is opposite all points of $T$. Suppose first, that $T$ contains (at least) two collinear points, say $x_1,x_2$. Let the line $L$ through $x_1$ and $x_2$ contain $\ell$ points of $T$. If $\ell=s+1$, then we are done, so assume $2\leq\ell\leq s$. Let $p$ be a point on $L\setminus T$. If there were some line $K$ through $p$ containing no point of $T$, then, since $x_1$ and $x_2$ project onto the same point $p$ of $K$, there exists a point on $K$ opposite each member of $T$, a contradiction. Hence \[\ell+t(s+1-\ell)\leq |T|=s+1,\] implying $t\leq 1$, a contradiction. 

Now suppose every pair of points in $T$ is opposite. Pick $x_1,x_2\in T$ arbitrarily. Let $z\in \{x_1,x_2\}^\perp$ be arbitrary and note $z\notin T$. Then, as in the previous paragraph, every line through $z$ contains at least (and hence exactly, by our assumption that no pair of points of $T$ is collinear) one point of $T$. Hence $s\geq t$. Set \[X=T\cup\{x_1^\perp\cap x_2^\perp\mid x_1,x_2\in T\}\] and let $\mathcal{L}$ be the set of intersections of $X$ with a line intersecting $X$ in at least two points. A moments' thought reveals that $X$, together with the line set $\mathcal{L}$, is an ideal subquadrangle $\Gamma$, say with order $(s',t)$. Clearly, $T$ is an ovoid in $\Gamma$. This implies that $s't=s$, so $s'=s/t$. 

If, in the last case, $s=t$, then $\Gamma$ has order $(1,t)$, implying that $X=(x_1^\perp\cap x_2^\perp)\cup \{x_1,x_2\}^{\perp\perp}$, with $x_1,x_2$ two arbitrary members of $T$. Clearly an ovoid of $X$ containing $x_1,x_2$ is $\{x_1,x_2\}^{\perp\perp}$.  
\end{proof}
We now show that the ovoid case in the previous lemma really leads to examples of minimal blocking sets.. 
\begin{lem}\label{ovoidcaseproof}
Suppose $\Delta$ is a generalised quadrangle of order $(s,t)$ with $s,t\geq 2$. Suppose $\Gamma$ is a subquadrangle of order $(s/t,t)$ and suppose also that $\Gamma$ has an ovoid $O$. Then $|O|=1+s$ and no point of $\Delta$ is opposite all members of $O$ and each point of $\Delta$ not in $\Gamma$ is collinear to a unique member of $O$. 
\end{lem}

\begin{proof}
By 1.8.1 of \cite{FGQ}, we have $|O|=1+(s/t)t=1+s$. Clearly each point of $\Gamma$ is collinear to $1+t$ members of $O$. Let $x$ be a point of $\Delta$ not in $\Gamma$. By 2.2.1 of \cite{FGQ}, $x$ is collinear to $1+s/t$, hence at least one, points of $\Gamma$. Since $\Gamma$ is ideal in $\Delta$, we readily deduce that $x$ is contained in a unique line of $\Delta$ that contains a line $L$ of $\Gamma$ and the points of $\Gamma$ collinear to $x$ are precisely the points of $L$. Since $O$ is an ovoid of $\Gamma$, we have $|O\cap x^\perp|=|O\cap L|=1$ and the lemma is proved. 
\end{proof}
\begin{remark}\label{ovoidcase}
Suppose in \Cref{quadranglepoints}, $\Delta$ is Moufang,  and $T$ is not a line. Then $\Delta$ has order $(q,q)$, or $(q^2,q)$ or $(q^3,q^2)$. In the latter case, $T$ is an ovoid in a subquadrangle of order $(q,q^2)$, which contradicts 1.8.3 of \cite{FGQ}. In case the order is $(q,q)$, we   note that the only Moufang quadrangles of order $(q,q)$ with large hyperbolic lines are the symplectic quadrangles, as follows from Theorem~1.4 of \cite{Bon-Cuy-Mal:96}. Finally, let $\Delta$ be a Moufang quadrangle of order $(q^2,q)$. Then any subquadrangle of order $(q,q)$ is a symplectic polar space (see \cite[\S3.5]{FGQ}) and, by 3.4.1$(i)$ of \cite{FGQ}, admits ovoids if, and only if, $q$ is even. 
\end{remark}

For proving \cref{A} for maximal singular subspaces, the parabolic quadrics (which correspond to the split buildings of type $\mathsf{B}_n$), play an exceptional role. That is why we treat them first.  

\begin{lem}\label{MSSparabolic}
If in a parabolic quadric $Q$ of Witt index $n$, $n\geq 2$, every line contains precisely $s+1$ points, then every set $T$ of $s+1$ maximal singular subspaces admits an opposite maximal singular subspace, except if they form a line or hyperbolic line in the corresponding dual polar space $\Gamma(Q)$. 
\end{lem}
\begin{proof}
Embed $Q$ naturally in a hyperbolic quadric $Q'$ of Witt index $n+1$. Each member $M$ of $T$ is contained in a unique maximal singular subspace $M'$ of $Q'$ of given type. The set of all such subspaces $M'$ admits an opposite maximal singular subspace $W'$, if they do not form a line in the corresponding half spin geometry, that is, by \cite{Kas-Mal:13}, if $T$ is not a line or a hyperbolic line in $\Gamma(Q)$. Then $W=W'\cap Q$ is a maximal singular subspace of $Q$ disjoint from every member of $T$.
\end{proof}

Now we can treat the other cases. It is convenient to first collect the rank $2$ case from what we already showed above. 

\begin{lem}\label{quadranglelines}
If every line of a polar space $\Delta$ of rank $2$ contains exactly $s+1$ points, and every point is contained in exactly $t+1$ lines,  $t>1$, then there exists a line disjoint from each member of an arbitrary set $T$ of $t+1$ (distinct) lines, except if these lines either contain a common point, or form a spread in a subquadrangle of order $(s,t/s)$. 
\end{lem}

\begin{proof}
This is just the dual of \cref{quadranglepoints}.
\end{proof}
We can now show the general case.

\begin{lem}\label{MSSgeneral}
Let $\Delta$ be a finite polar space of rank $n\geq 3$ and of order $(s,t)$, with both $s$ and $t$ at least $2$. Suppose $\Delta$ does not correspond to a parabolic quadric. Then  every set $T=\{V_0,V_1,\ldots,V_t\}$ of $t+1$ maximal singular subspaces admits an opposite maximal singular subspace, except if they contain a common singular subspace $U$ either of codimension $1$ (that is, they form a line in the corresponding dual polar space $\Gamma(Q)$), or of codimension $2$, $t=s^2$, and they form a spread in a symplectic subquadrangle of order $(s,s)$, $s$ even, of $\Res_\Delta(U)$ (and then $\Delta$ corresponds to an elliptic quadric).
\end{lem}

\begin{proof}
We prove this by induction on the rank $n$ of $\Delta$, where \cref{quadranglelines} serves as our base, noting that a parabolic quadric is characterised by $s=t$ and admitting a full embedded grid. In particular, if $s=t$, then the case in  \cref{quadranglelines} of a spread in a subquadrangle does occur.

An easy count yields a point $x$ not in any of the subspaces $V_i$, $i=0,1,\ldots,t$. In $\Res_\Delta(x)$ we can now consider the maximal singular subspaces $V_i':=\<x,x^\perp\cap V_i\>$, $i=0,1,\ldots,t$. A maximal singular subspace through $x$ locally opposite all of those at $x$ is opposite every member of $T$. Hence, by induction, there are two possibilities. 
\begin{compactenum}[$(1)$]
\item  \emph{The intersection of all $V_i'$ is a subspace $W$ of dimension $n-2$.}\\ Set $U_i:=V_i\cap W$. If $U_0=U_1=\cdots=U_t$, then we can apply induction in $\Res_\Delta(U_0)$. So, without loss of generality, we may assume for a contradiction that $U_1\neq U_2$. Then we select an arbitrary point $y\in V_0'\setminus W$. The subspaces $\<y,y^\perp\cap V_0\>=V_0'$ and $\<y,y^\perp\cap V_1\>\supseteq\<y,U_1\>$ intersect in the $(n-2)$-dimensional singular subspace $\<y,U_1\>$ and hence induction implies that $\<y,U_2\>=\<y,U_1\>$. Consequently, $U_1=\<y,U_1\>\cap W=\<y,U_2\>\cap W=U_2$, a contradiction.  
\item \emph{The intersection of all $V_i'$ is a subspace $W$ of dimension $n-3$, and $W$ is also the pairwise intersection of the $V_i$.}\\ Here, the same arguments as in the previous paragraph $(1)$ do the job, except if $n=3$, since in this case $W=\{x\}$. So, we have to do the case $n=3$ separately and explicitly (in another way). In this case, $\Delta$ arises from an elliptic quadric, $t=s^2$, $s$ is even, all planes $V_i'$ are locally opposite each other and form a spread in a subquadrangle of order $(s,s)$ of $\Res_\Delta(x)$. Suppose for a contradiction that the members of $T$ are not pairwise opposite and let $V_0$ not be opposite $V_1$. Then an easy count yields a point $x'\perp V_0\cap V_1$ not contained in any member of $T$. Switching the roles of $x$ and $x'$ brings us back to Case (1). Hence all members of $T$ are pairwise opposite, that is, disjoint. Let $L_0$ be any line of $V_0$ and let $\{p_1\}=L_0^\perp\cap V_1$. If some point $x''\in \<p_1,L_0\>$ does not belong to any member of $T$, then we can switch the roles of $x$ and $x''$ and this brings us back to Case (1). Hence every point of $\<p_1,L_0\>$ belongs toi some member of $T$. Varying $L_0$ in $V_0$, we obtain $(s^2-1)(s^2+s+1)$ points, which then must cover all points of $V_2\cup V_3\cup\cdots\cup V_t$. If we view $\Delta$ as an elliptic quadric in $\PG(7,s)$, then $T$ is contained in the $5$-dimensional subspace $Y$ of $\PG(7,s)$ generated by the planes $V_0$ and $V_1$. Since $Y$ intersects $\Delta$ in a hyperbolic polar space of rank $3$, and such a polar space does not contain three pairwise opposite planes, this leads to a contradiction.
\end{compactenum}
The proof is complete.
\end{proof}

We leave it to the reader to check that the second case of the previous lemma does occur, that is, if  $\Delta$ corresponds to an elliptic quadric, a set $T$ of $t+1$ generators containing a common singular subspace $U$ of dimension $n-3$ forms a spread in a symplectic subquadrangle of order $(s,s)$, $s$ even, of $\Res_\Delta(U)$, then no generator of $\Delta$ is opposite each member of $T$. Likewise, the second case of the next lemma also really occurs. 

\begin{lem}
Let $\Delta$ be finite polar space of rank $n\geq 3$, order $(s,t)$, and suppose $\Delta$ is not symplectic. Let $T=\{\alpha_{0}, \dots, \alpha_{s}\}$ be a set of $s+1$ different singular subspaces of common dimension $\ell\leq n-2$ in $\Delta$, with $\ell\neq n-2$ if $t=1$, such that they do not form a pencil in a singular space of dimension $\ell+1$ (and we also assume $\ell\geq 1$) and such that, if $\Delta$ is a small Hermitian polar space and $\ell=n-2$, $T$ is not an ovoid in a symplectic subquadrangle of order $(t,t)$ in the residue of some $(\ell-1)$-dimensional subspace of $\Delta$ contained in each member of $T$. Then there exists a singular subspace opposite all of $\alpha_{0}, \dots, \alpha_{s}$ in $\Delta$.
\end{lem}

\begin{proof}
Let, for given order $(s,t)$ and rank $n$, the polar space $\Delta$ and $\alpha_0,\ldots,\alpha_s$ be a smallest (with respect to $n-\ell$) counterexample to the lemma. 

We claim that there exists a singular subspace $\beta$ of dimension $\ell+1$ such that $\proj_\beta(\alpha_i)$ is a point $p_i$, for all $i\in\{0,1,\ldots,s\}$. Indeed, first assume $\ell\leq n-3$. We select arbitrarily singular subspaces $\beta_i$ of dimension $\ell+1$ containing $\alpha_i$, for all $i$. We can easily choose them in such a way that they do not have a subspace of dimension $\ell-1$ in common. The induction hypothesis implies that we find a singular subspace $\beta$ of dimension $\ell$ opposite each $\beta_i$, $i=0,1,\ldots,s$. Then $\beta$ satisfies the claim. Now suppose $\ell=n-2$, and note that in this case we assume $t>1$ (as submaximal singular subspaces do not belong to the oriflamme geometry). We proceed by induction on the rank $n$, including the case of rank $2$. If $n=2$, then the number of lines not disjoint from $T$ is at most $(s+1)(t+1)$. Since there are $(1+t)(1+st)$ lines in total, and $t>1$, the claim follows. Now let $n\geq 3$. Clearly, we can find a point $x$ not contained in any member of $T$. Applying induction in $\Res(x)$ proves the claim. 


So, there is a unique point $p_i$ in $\beta$ collinear to all points of $\alpha_i$. If $\{p_0,p_1,\ldots,p_s\}$ is not a line, then by \cref{An} we find a subspace $\alpha\subseteq\beta$ of dimension $\ell$ not containing any of the points $p_0,p_1,\ldots,p_s$. Then $\alpha$ is opposite each member of $T$ and we are done.

So we may assume that $\{p_0,p_1,\ldots,p_s\}$ is a line $L$. The set of points of $\beta_i$ collinear to all points of $L$ is a subhyperplane $H_i$ contained in $\alpha_i$ (as a hyperplane of the latter).  

\textbf{Step I.} Suppose $\ell\leq n-3$. We claim that all points of $\alpha_i\setminus H_i$ are collinear to all points of $\alpha_j\setminus H_j$, for distinct $i,j\in\{0,1,\ldots,s\}$. Indeed, we may assume $(i,j)=(0,1)$. Let $x_0\in\alpha_0\setminus H_0$ and $x_1\in\alpha_1\setminus H_1$. Then $x_0^\perp\cap \beta=:K_0$ and $x_1^\perp\cap\beta=:K_1$ are two hyperplanes of $\beta$ none of which containing $L$, but containing $p_0$ and $p_1$, respectively. It follows that $N:=K_0\cap K_1$ is a subspace of dimension $\ell-1\geq 0$ disjoint from $L$.  The line $N^\perp\cap \beta_i$ intersects $\alpha_i$ in a unique point $x_i$ as otherwise $N$ belongs to the perp of a point of $H_i$, contradicting the facts that also $L$ belongs to the perp, that $L$ and $N$ are complementary, and $\beta$ and $\beta_i$ are opposite (and note that the notation $x_i$ is in conformity with the definitions of $x_0$ and $x_1$ above), We may hence view $x_0,x_1,\ldots.x_s$ as points of $\Res_\Delta(N)$. If they do not constitute a line in $\Res_\Delta(N)$, then by \cref{polarpoints} we can find a point opposite all of them, meaning using \cref{lemma}, we can find a subspace $\alpha$ of dimension $\ell$ containing $N$ opposite each of $\alpha_i$, $i\in\{0,1,\ldots,s\}$. Hence all of $x_0,x_1,\ldots,x_s$ are collinear and the claim is proved. 

The previous claim now easily implies that all members of $T$ are contained in a common singular subspace, say $M$, which we may assume to be maximal. Since by assumption, they do not form a Grassmann line in $U$, \cref{An} yields a subspace $U\subseteq M$ of dimension $n-2-\ell$ disjoint from all $\alpha_i$. Let $M'$ be a maximal singular subspace of $\Delta$ opposite $M$. Then $U^\perp\cap M$ has dimension $\ell$ and is opposite each member of $T$ (use \cref{lemma} again).

\textbf{Step II.} Now suppose $\ell=1$ and $n=3$. 
 Note that, if all elements of $S$ intersect in a common point $p$, our assumptions yield a line $M$ locally opposite every element of $S$ in $\Res(p)$. Then on $M$ we can choose a point $q \neq p$. Let $N$ be a line locally opposite $M$ in $\Res(q)$. Then the line $N$ is opposite each element of $S$.


Recall that the points $p_0,\ldots,p_s$ form a line $L$ in the plane $\beta$. Let $b$ be an arbitrary point in $\beta \setminus L$. The projection of $b$ onto $\alpha_{i}$ is a point $b_{i}$. We can find a line $N$ locally opposite all lines $bb_i$ at $b$, except if these lines form a line or an ovoid in a subquadrangle of order $(s/t,t)$ of $\Res(b)$. Then, in $\Delta$, the line $N$ is opposite all $L_{i}$.

Hence, for all $b\in\beta\setminus L$, we may suppose that the lines $bb_{i}$ form a line or an ovoid in a subquadrangle of order $(s/t,t)$ of $\Res(b)$. In the former case, we say that $b$ has type (L), in the latter case we say that $b$ has type (O). 
We consider four possibilities.

\begin{compactenum}[$(i)$]
\item 
If every point of $\beta\setminus L$ has type (L), then, as in Step I, all points of $\alpha_i\setminus H_i$ are collinear to all points of $\alpha_j\setminus H_j$, for all $i,j\in\{0,1,\ldots,s\}$. This easily implies that all members of $T$ are contained in a unique plane $\gamma$. Clearly, each member of $T$ contains $L^\perp\cap\gamma$, leading to $T$ being a line pencil in $\gamma$, a contradiction.

\item If each point of $\beta\setminus L$ has type (O), then no point of $L_i\setminus H_i$ is collinear to any point of $L_j\setminus H_j$, $i\neq j$, and so, each point $H_i$ is collinear to each line $L_j$. In particular, all $H_i$ are collinear. Since they are collinear to $L$, they lie in the same plane $\delta$ with $L$. 

Since $L_i$ is collinear to $p_i,H_i$ and $H_j$, and not to $p_j$, $i\neq j$, we deduce that $H_j\in\<p_i,H_i\>$, Switching the roles of $i$ and $j$, we find $H_i=\<p_i,H_i\>\cap\<p_j,H_j\>=H_j$. Hence all members of $T$ share a common point and the first paragraph of Step II concludes this case. 



\item Next, suppose that every point in $\beta \setminus L$, except one (say $b$), has type (O). Note that by \cref{ovoidcase} $s=t^2\geq 4$. We consider $L_0$ and $L_1$ and set $b_i=b^\perp\cap L_i$, $i=0,1$. Then $b_0\perp b_1$. As in $(ii)$, all other points of $L_0\setminus H_0$ have to be collinear to $H_1$, implying, since $s\geq 4$, that $H_1\perp L_0$. In particular $H_1\perp b_0$ and so $b_0\perp L_1$. Likewise $H_0\perp L_1$ and $b_1\perp L_0$. hence $L_0\perp L_1$ contradicting the existence of points of type (O). Hence this case cannot occur.

\item Lastly, we consider the case that at least one point $d\in\beta \setminus l$ has type (O) and at least two points $b,c\in\beta\setminus L$ have type (L). Again $s=t^2\geq 4$.  Set $b_i=b^\perp\cap L_i$, $c_i=c^\perp\cap L_i$ and $d_i=d^\perp\cap L_i$, $i=0,1,\ldots,s$. Suppose, without loss, that $bc\cap L=p_0$. Then $b_0=c_0$ and $b_0\perp L_i$, $i=1,2,\ldots,s$. So every point of $bc\setminus\{p_0\}$ has type (L). This argument also implies that every point of $\beta\setminus(bc\cup L)$ has type (O). This now implies that each point of $L_0\setminus\{b_0\}$ is only collinear to the point  $H_i$ of $L_i$. Hence $H_i\perp L_0$ implying $H_0\perp H_i$, for all $i=1,2,\ldots,s$. Hence all $H_i$ are contained in the same plane $\delta$ with $L$. Since $b_i\perp b_j$, but $b_i$ is not collinear to $c_j$, we have $H_i\neq H_j$, $i,j\in\{1,2,\ldots,s\}$, $i\neq j$. Also, since $L_0\perp p_0,H_i$, for all $i$, all $H_i$ are contained in one line $H$ of $\delta$. Since $H\setminus\{p_0\}$ contains $s$ points, we may without loss assume that $H_0=H_1$. But then $L_0\perp L_1$, a contradiction. Hence this case does not occur and the proof of Step II is complete.
\end{compactenum}

\color{black}
\textbf{Step III.} Finally suppose $\ell=n-2$ and $n\geq 4$.  Let $U$ be an $(n-4)$-dimensional subspace of $\beta$ disjoint from $L$ and let $W\supseteq U$ be an $(n-3)$-dimensional subspace of $\beta$, also disjoint from $L$. If $A:=\{\<U,U^\perp\cap \alpha_i\>\mid i=0,1,\ldots,s\}$, considered as a set of lines of $\Res_\Delta(U)$, is neither a planar line pencil nor a set of lines through some point $x$ forming an ovoid in a symplectic ideal subquadrangle of the point residual at $x$, then we can select an $(n-2)$-dimensional subspace $\alpha^*$ through $U$ locally opposite each member of $A$. Then $\alpha^*$ is opposite each member of $T$ (using \cref{lemma} again). Hence we may assume that $A$ is either a planar line pencil (we say $A$ is of type (PLP)), or a a set of lines through some point $x$ forming an ovoid in a symplectic ideal subquadrangle of the point residual at $x$ (and we say $A$ is of type (OSS)). In the former case, the set $B:=\{\<W,W^\perp\cap \alpha_i\>\mid i=0,1,\ldots,s\}$ is a line in the generalised quadrangle $\Res_\Delta(W)$ (and we say that $W$ is of type (L)); in the latter case $B$ is an ovoid in a symplectic ideal subquadrangle of $\Res_\Delta(W)$ (and we say $W$ is of type (O)). It follows that, in the former case all subspaces of $W$ of dimension $n-4$ are of type (PLP), and in the latter case all such subspaces are of type (OSS). Now varying $W$ we easily deduce that all $(n-2)$-dimensional subspaces of $W$ disjoint from $L$ are either of type (L), or of type (O). Now we continue just like in Step II and arrive at contradictions. 

This concludes the proof of the lemma.  
\end{proof}

It remains to consider symplectic polar spaces. It turns out that type $2$ elements, that is, lines, cannot be included in the general proof. We will treat them separately. However, the final proof is inductive and the result for lines in rank $3$ is needed to prove the general case; which is then used to prove the result for lines in higher rank. This explains the rather peculiar conditions in the next lemma, which shall become clear in the proof of \cref{patch} below.

\begin{lem}\label{symplines}
Let $\Delta$ be a symplectic polar space of rank at least $3$ and order $(s,s)$. Suppose that every set of $s+1$ singular planes admits a common opposite plane, except if they all intersect in a line and are contained in a common $3$-space of the underlying projective space. Then a set $T$ of $s+1$ lines of $\Delta$ admits a common opposite in $\Delta$ if, and only if,  $T$ is not a line pencil in some plane of the underlying projective space.
\end{lem}

\begin{proof}
Let $T$ be a set of $s+1$ lines of the symplectic polar space $\Delta$ of rank $r$ naturally embedded in $\PG(2r-1,s)$, $r\geq 3$. Suppose $T$ is not a line pencil in some plane of $\PG(2r-1,s)$. We show that there exists a line of $\Delta$ opposite all members of $T$. As usual, we set $T=\{L_0,L_1,\ldots, L_s\}$. 

We include $L_i$ in a plane $\alpha_i$ in such a way that the $\alpha_i$ do not contain a common line (which can be easily accomplished). Let $\beta$ be a plane opposite all $\alpha_i$, $i\in\{0,1,\ldots,s\}$. Set $m_i:= L_i^\perp\cap\beta$. If $\{m_i|i\in\{0,1,\ldots,s\}\}$ is not a line, then we can find a line $L$ in $\beta$ not containing any of the $m_i$ and hence opposite all of the $L_i$. So the $m_i$ constitute a line $M$. Let $b\in\beta\setminus M$ be arbitrary. Then $b_i:=b^\perp\cap L_i$ is a unique point. Suppose the lines $bb_i$ do not form a line pencil in a plane of $\PG(2r-1,s)$. Then we can find a line $L$ through $b$ locally opposite all of the $bb_i$.  Then $L$ is opposite all of the $L_i$ by \cref{lemma}. Hence we may assume that the lines $bb_i$ form a line pencil in a plane $\pi_b$ of $\PG(2r-1,s)$. Suppose now that for two choices of $b\in \beta\setminus M$, the points $b_i$ are contained in a common line $K_b$ of $\PG(2r-1,s)$. Let $b,c$ be those two points and adapt the same notation for $c$ as we introduced for $b$. Without loss of generality, we may assume that the line $bc$ contains the point $m_0$. Then the lines $L_i$, $i\in\{1,2,\ldots,s\}$ are contained in the plane $\gamma$ of $\PG(2r-1,s)$ spanned by $K_b$ and $K_c$. The point $b_0=c_0$ is also contained in $\gamma$. 

Suppose first that also $L_0$ is contained in $\gamma$. Then, since we assumed that $T$ is not a line pencil, it is easy to see that $\gamma$ is a singular plane, and it contains a point $x$ not on any of the $L_i$. Then the line $x^\perp\cap \gamma'$, with $\gamma'$ a plane opposite $\gamma$ in $\Delta$ is opposite each member of $T$. 

Suppose now that $L_0$ is not contained in $\gamma$. Let $z:=L_1\cap L_2$. Then $z\perp b_0$ as $z\perp\{b_1,b_2\}\subseteq K_b\ni b_0$. We can select a singular plane $\beta_z$ containing $zb_0$ and such that $L_0$ is not collinear to $\beta_z$, and $\beta_z$ is not in a common singular $3$-space with $\gamma$ if singular. 

In $\beta_z\setminus zb_0$ we can find a point $y$ not collinear to $b_1$. It follows that $y$ is not collinear to any of the $b_i$, $i\in\{1,2\ldots,s\}$. The lines joining $y$ with the unique projection point of $y$ on the $L_i$ do not form a line pencil in any plane as two of these lines coincide (namely, the line $yz$ joins $y$ with the point $z$ of both $L_1$ and $L_2$). Hence we can find a line $L$ through $y$ locally opposite all these lines. Again, by \cref{lemma}, $L$ is opposite each member of $T$. 

Hence we may assume that for at most one point $b\in\beta\setminus M$, the points  $b_0,b_1,\ldots,b_s$ are on one line (and we denote that point, if it exists, from now on with $b^*$). Note that we may also assume that $L_0$ and $L_1$ do not intersect. Indeed, if all members of $T$ pairwise intersect, then either they are contained in a plane, and we treated that case above, or they all contain a common point $p$. In the latter case the result follows from considering the residue at $p$ (indeed, we can then select a line $K$ locally opposite each member of $T$; then select a point $q\in K\setminus\{p\}$ and a line $L$ locally opposite $K$ at $q$. The line $L$ is opposite each member of $T$ by \cref{lemma}).

First let $s>2$. Choose points $b,c\in\beta\setminus M$ with $m_2\in bc$ and $b^*\notin\{b,c\}$. As before, the lines $L_0$ and $L_1$ are contained in $\<\pi_b,\pi_c\>$, which is a $4$-space $U_2$ of $\PG(2r-1,s)$. If we choose $d\in\beta\setminus M$ with $m_3\in bd$ and $d\neq b^*$, then we obtain likewise that $L_0$ and $L_1$ are contained in a $4$-space $U_3$ of $\PG(2r-1,s)$ spanned by $\pi_b$ and $\pi_d$. Suppose first these two $4$-spaces coincide.  Then $U_2=U_3$ contains $\beta$. The polar space induced in $U_2$ is degenerate, but, as we have opposite lines ($L_0$ and some line in $\beta$), the radical is a point, which must coincide with the intersection of any pair of singular planes, contradicting $m_0\neq m_1$. Hence $U_2\neq U_3$ and these intersect in a $3$-space containing $L_0,L_1$ and $b$. However, $L_0$ and $L_1$ already span $U_2\cap U_3$, and by the arbitrariness of $b$, the $3$-space $\<L_0,L_1\>$ contains $\beta$, which is ridiculous as $\beta$ is disjoint from $L_0\cup L_1$. 

Now let $s=2$. Similar arguments as in the previous paragraph show that $W:= \<L_0,L_1,L_2\>$ contains $\beta$. Then again $\dim W=3$ leads to the contradiction that $L_0\cap \beta$ is non-empty and if $\dim W=4$, then we have a degenerate symplectic polar space induced in $W$, leading to the same contradiction as before. Hence $\dim W=5$. Also as before, $W$ is a non-degenerate symplectic polar space. We coordinatise $W$ as follows (using obvious shorthand notation). The two points on $L_0$ not collinear to all points of $M$ are labelled $100000$ and $010000$. Likewise those on $L_1$ and $L_2$ by $001000$, $000100$ and $000010$, $000001$, respectively. We may assume that the points $100000$, $001000$ and $000010$ are together in a plane $\pi_b$, $b\in \beta$, and then $b$ has labels $101010$. Similarly we have the points $100101$, $011001$ and $010110$ in $\beta\setminus M$. It follows that $m_0=001111$, $m_1=110011$ and $m_2=111100$. Then the point $111111$ is contained in each of the planes $\<m_i,L_i\>$, $i=0,1,2$, and hence collinear in $\Delta$ with all points of $L_0\cup L_1\cup L_2$, which spans $W$, contradicting non-degeneracy.

The proof is complete.
\end{proof}

\begin{lem}\label{sympi}
Let $\Delta$ be a symplectic polar space of rank  $r\geq 4$ and order $(s,s)$, and let $i\in\N$ be such that $2\leq i\leq r-2$. Suppose that every set of $s+1$ singular subspaces of dimension $i+1$ admits a common opposite, except if they all intersect in an $i$-dimensional singular subspace and are contained in a common $(i+2)$-space of the underlying projective space. Suppose also that \emph{\cref{A}} is true for symplectic polar spaces of rank at most $r-1$. Then a set $T$ of $s+1$ $i$-dimensional singular subspaces of $\Delta$ admits a common opposite in $\Delta$ if, and only if, all members $T$ are not contained in a common $(i+1)$-dimensional subspace of the underlying projective space if they contain a common $(i-1)$-dimensional singular subspace of $\Delta$. 
\end{lem}

\begin{proof}
Let $T$ be a set of $s+1$ singular subspaces of dimension $i$ of the symplectic polar space $\Delta$ of rank $r$ naturally embedded in $\PG(2r-1,s)$, $r\geq 4$. Suppose all members of $T$ are not contained in an $(i+1)$-dimensional subspace of $\PG(2r-1,s)$ if they share a common $(i-1)$-dimensional singular subspace of $\Delta$.  We show that there exists a singular $i$-dimensional subspace of $\Delta$ opposite all members of $T$. We set $T=\{U_0,U_1,\ldots, U_s\}$.

First we assume that all members of $T$ are contained in a common $(i+1)$-dimensional subspace $W$ of $\PG(2r-1,s)$. By our assumption, not all members of $T$ share the same $(i-1)$-dimensional subspace. This implies that the radical of the polar space induced in $W$ has dimension strictly larger than $i-1$; hence $W$ is a singular subspace. Our assumption and \cref{An} imply that we can find a point $x\in W$ not contained in any member of $T$. Let $W'$ be a singular subspace of dimension $i+1$ opposite $W$. Then $x^\perp\cap W'$ is a singular subspace of dimension $i$ opposite each member of $T$.

Henceforth we may assume that not all members of $T$ are contained in the same subspace of $\PG(2r-1,s)$ of dimension $i+1$.
 
 Include every member of $T$ in a singular subspace of dimension $i+1$ in such a way that not all of them are contained in a common subspace of dimension $i+2$ of $\PG(2r-1,s)$. By assumption we can find a subspace $W$  opposite all of these $(i+1)$-dimensional subspaces. Set $\{m_i\}=W\cap U_i^\perp$. If the $m_i$ do not form a line, then \cref{An} yields a singular $i$-space $U\subseteq W$ not containing any of the $m_i$ and hence opposite all members of $T$. Hence we may assume that the $m_i$ form a line $M$. Set $H_i=U_i\cap M^\perp$. Then $H_i$ is a hyperplane of $U_i$. 
 
Assume now that two members of $T$, say $U_0$ and $U_1$, do not share a common $(i-1)$-dimensional singular subspace. Set $D=U_0\cap U_1$. Let $H$ be a hyperplane of $U_0$ not containing $D$, and hence distinct from $H_0$. Then $K:=H^\perp\cap W$ is a line through $m_0$. Pick $x\in K\setminus\{m_0\}$ arbitrarily. Define $U^x_i:=\<x,x^\perp\cap U_i\>$. Suppose there exists an $i$-dimensional singular subspace $U$ through $x$ locally opposite all $U_i^x$. Then, by \cref{lemma}, $U$ is opposite each member of $T$. Our hypotheses imply that the $U_i^x$ are contained in an $(i+1)$-dimensional subspace $A_x$ of $\PG(2r-1,s)$ and all $U_i^x$ share a common $(i-1)$-space $B_x$. The singular subspace $B_x$ contains a hyperplane of $H$ and a hyperplane of $x^\perp\cap U_1$. As $D$ is not contained in $H$, these two hyperplanes do not coincide, and hence they generate $B_x$. Now we do the same construction with $y\in K\setminus\{x,m-)\}$ and obtain the similarly defined subspaces $A_y$ and $B_y$. It is elementary to check that $x^\perp\cap U_i\neq y^\perp\cap U_i$, for $i\in\{1,2,\ldots,s\}$. Hence $A:=\<A_x,A_y\>$ contains $U_i$ for all $i\in\{1,2,\ldots,s\}$.  The intersection $A_x\cap A_y$ contains $H$. Also, since $H$ does not contain $D$, it does not contain  the intersection $x^\perp\cap y^\perp\cap U_1$ (which is a singular subspace of dimension $(i-2)$).  It follows that $A_x$ and $A_y$ share at least a subspace of dimension $i$ (generated by $H$ and $x^\perp\cap y^\perp\cap U_1$). Hence $\dim A\in \{n+1,n+2\}$. Suppose $\dim A=n+1$. Then $A_x=A_y$ and contains the line $\<x,y\>$, which necessarily intersects $U_1$ for dimension reasons. But this contradicts the fact that $W$ is opposite some singular $(i+1)$-space containing $U_1$. Consequently $\dim A=n+2$. We can now do the same thing with another hyperplane $H'$ of $U_0$ not containing $D$ and distinct from $H_0$, and obtain the similarly defined subspace $A'$ of dimension $n+2$ and the line $\<x',y'\>$ of $W$, with $m_0\in\<x',y'\>$. Since both $A$ and $A'$ share the subspace $\<H,U_1\>$ of dimension at least $i+1$, we have $\dim\<A,A'\>\in \<n+2,n+3\>$. If $A=A'$, then $A$ contains the plane $\<x,x',m_0\>$, which necessarily intersects $U_1$ in at least a point since $\dim A=n+2$ and $\dim U_1=i$. This is again a contradiction. It follows that $A_0:=A\cap A'$ has dimension $i+1$ and contains all of $U_1,U_2,\ldots, U_s$, plus $H$ and $m_0$. If $A_0$ were singular, then $m_0=m_1$, a contradiction. Hence $A_0$ is not singular and all $U_i$, $i\in\{1,2,\ldots,s\}$, share a common $(i-1)$-space $V_0$. There is a unique $i$-space $U^*_0$ through $V_0$ in $A_0$ distinct from $U_i$ for any $i\in\{1,2,\ldots,s\}$.     The space $U^*_0$ contains $H$ and $m_0$. 
 
 Now we can find a singular $(i+1)$-space $W^*$ containing $U^*_0$ with the property that not all of its points are collinear to all points of $U_0$. In $W^*$, we can then find a point $z$ not collinear to all points of $U_1$ (because $U_1^\perp$ cuts out a hyperplanes of $W^*$). Since a point of $\Delta$ is collinear to either all points, or a hyperplane of points of $W^*$,  we see that $z^\perp$ intersects each $U_i$ in a hyperplane of $U_i$, and for $i=1,2,\ldots,s$, that hyperplane is necessarily $V_0$. Hence there exists an $i$-space $U$ through $z$ locally opposite the two spaces $\<z,H\>$ and $\<z,V_0\>$ at $z$, and $U$ is opposite each member of $T$ by \cref{lemma}. 
 
 Hence we may assume that each pair in $T$ intersects in an $(i-1)$-space. Then either that $(i-1)$-space is unique, say $V^*$, or all members of $T$ are contained in some $(i+1)$-space, say $W^*$. In the former case, our hypotheses permit to find an $i$-space $U$ through $V^*$ locally opposite each member of $T$, and then the projection $U'$ of $U$ onto a singular subspace opposite $V^*$ is opposite each member of $T$ at $V^*$. In the latter case we are back to the situation of the first paragraph of this proof, which we already handled. 
 
 This completes the proof of the lemma. 
\end{proof}

We can now prove \cref{A} for symplectic polar spaces. 

\begin{prop}\label{patch}
Let $\Gamma$ be a symplectic polar space of rank $r$ at least $2$ and order $(s,s)$. Let $T$ be a set of  $s+1$ singular subspaces of $\Gamma$ of dimension $k\leq r-2$. Then there exists a singular subspace of dimension $k$ opposite each member of $T$, except if either $T$ is a line of the corresponding $k$-Grassmannian geometry, or all members of $T$ contain a given $(k-1)$-dimensional subspace in the residue of which they form a hyperbolic line.
\end{prop}

\begin{proof}
We prove this by induction on $r$, and for given $r$ we use induction on $r-i$. 

Let first $r=3$. Then $i=1$ and the assertion holds by  \cref{symplines}, \cref{MSSgeneral} and \cref{MSSparabolic} (we need the latter in characteristic $2$). 

Now assume $r\geq 4$. Then the assertion holds for $i\geq 2$ by \cref{sympi}, and it holds for $i=1$ by \cref{symplines}. 

The proof is complete. 
\end{proof}

Now \cref{ovoidcase} yields \cref{B}. 

\end{document}